\documentclass[journal=jnw]{CUP-JNL-DTM}%

\usepackage{graphicx}
\usepackage{multicol,multirow}
\usepackage{amsmath,amssymb,amsfonts}
\usepackage{mathrsfs}
\usepackage{amsthm}
\usepackage{rotating}
\usepackage{appendix}
\usepackage{ifpdf}
\usepackage[T1]{fontenc}
\usepackage{newtxtext}
\usepackage{newtxmath}
\usepackage{textcomp}
\usepackage{xcolor}
\usepackage{lipsum}
\usepackage[colorlinks,allcolors=blue]{hyperref}

\newtheorem{theorem}{Theorem}[section]
\newtheorem{lemma}[theorem]{Lemma}
\theoremstyle{definition}
\newtheorem{corollary}[theorem]{Corollary}
\newtheorem{proposition}[theorem]{Proposition}
\newtheorem{remark}[theorem]{Remark}

\numberwithin{equation}{section}

\newcounter{count}

\jname{A plethora of solitary waves for FDKP}
\articletype{}
\jyear{2025}

\newcommand{\diff}{\mathrm{d}}
\newcommand{\Diffb}{\mathbf{D}}
\newcommand{\Diff}{\mathrm{D}}

\newcommand{\dk}{\, \diff \mathbf{k}}
\newcommand{\dx}{\, \diff x}
\newcommand{\dy}{\, \diff y}

\newcommand{\FF}{\mathcal F}
\newcommand{\GG}{\mathcal G}
\newcommand{\LL}{\mathcal L}
\newcommand{\WW}{\mathcal W}

\renewcommand{\e}{\mathrm{e}}
\renewcommand{\i}{\mathrm{i}}

\newcommand{\bfk}{\mathbf{k}}

\begin{document}

\begin{Frontmatter}

\title[Article Title]{A plethora of fully localised solitary waves for the full-dispersion Kadomtsev--Petviashvili equation}

\author[1]{Mats Ehrnstr\"{o}m}
\author[2]{Mark D. Groves}

\authormark{Ehrnstr\"{o}m and Groves}

\address[1]{\orgdiv{Department of Mathematical Sciences}, \orgname{Norwegian University of Science
and Technology}, \orgaddress{\city{7491 Trondheim},  \country{Norway}}}

\address[2]{\orgdiv{Fachrichtung Mathematik}, \orgname{Universit\"{a}t des Saarlandes}, \orgaddress{Postfach 151150, \city{66041 Saarbr\"{u}cken}, \country{Germany}}. \email{groves@math.uni-sb.de}}




\abstract{The KP-I equation
arises as a weakly nonlinear model equation for gravity-capillary waves {with} Bond number
$\beta>1/3$, also called strong surface tension. This equation has recently
been shown to have a family of nondegenerate, symmetric `fully localised' or `lump' solitary waves which decay to zero in all spatial directions. The full-dispersion KP-I equation
is obtained by retaining the exact dispersion relation in the modelling from the water-wave problem.
In this paper we show that the FDKP-I equation also has a family of symmetric fullly localised solitary waves which are obtained by
casting it as a perturbation of the KP-I equation and applying a suitable variant of the
implicit-function theorem.}

\end{Frontmatter}



\section{Introduction}\label{sec:intro}

\subsection{The KP and FDKP equations}
The full-dispersion Kadomtsev--Petviashvili (FDKP) equation
\begin{equation}\label{FDKP}
u_t + m(\Diffb) u_x + 2 u u_x  = 0,
\end{equation}
where the Fourier multiplier $m$ is given by
\[m(\Diffb) = \left( 1 + \beta |\Diffb|^2 \right)^{\frac{1}{2}} \left( \frac{\tanh  |\Diffb|}{|\Diffb|} \right)^{\frac{1}{2}} \left( 1 + \frac{ 2\Diff_2^2}{\Diff_1^2} \right)^{\frac{1}{2}}\]
with $\Diffb = -\i(\partial_x, \partial_y)$, was introduced by \citet{Lannes} (see also \citet{LannesSaut14}) as an alternative to the classical
Kadomtsev--Petviashvili (KP) equation
\begin{equation}\label{KP}
(\zeta_t-2\zeta\zeta_x+\tfrac{1}{2}(\beta-\tfrac{1}{3})\zeta_{xxx})_x -\zeta_{yy}=0,
\end{equation}
which arises as a weakly nonlinear approximation for three-dimensional
gravity-capillary water waves. The parameter $\beta>0$ measures the relative strength of surface tension; the case $\beta>\frac{1}{3}$
for strong surface tension is termed KP-I, {while} the case $\beta<\frac{1}{3}$ for weak surface tension is KP-II. The analogous convention {is used} for the full-dispersion FDKP equation, giving us {an} FDKP-I equation for the strong surface tension case {studied in this paper}.

An FDKP solitary wave is a nontrivial, evanescent solution of \eqref{FDKP} of the form $u(x,y,t)=u(x-ct,y)$
with wave speed $c>0$, that is, a localised solution of the equation
\begin{equation}\label{steady FDKP}
-c u + m(\Diffb)u +  u^2 = 0.
\end{equation}
Similarly, a KP solitary wave is a nontrivial, evanescent solution of \eqref{KP} of the form\linebreak
$\zeta(x,y,t)=\zeta(x-\tilde{c}t,y)$
with wave speed $\tilde{c}>0$, that is, a localised solution of the equation
\begin{equation}\label{steady KP}
(\tilde{c} -1)\zeta  + \tilde{m}(\Diffb)\zeta +  \zeta^2 = 0,
\end{equation}
where
\[
\tilde{m}(\Diffb) = 1+\frac{\Diff_2^2}{\Diff_1^2} + \tfrac{1}{2}(\beta- \tfrac{1}{3})\Diff_1^2.
\]
{ Let us emphasise} that these waves are fully localised, that is, decaying in all spatial directions. The KP equation allows a scaling,
{ such that} the wave speed $\tilde{c}$ can be normalised to unity by the transformation
$\zeta(x,y) \mapsto \tilde{c}\zeta(\tilde{c}^\frac{1}{2}x,\tilde{c}y)$, which
converts \eqref{steady KP} into the equation
\begin{equation}
\tilde{m}(\Diffb)\zeta +  \zeta^2 = 0. \label{normalised steady KP}
\end{equation}

While it is known that the KP-II equation does not admit any 
solitary waves \citep{deBouardSaut97b}, the situation is rather different for the KP-I equation. Letting $\zeta(x,y)=\zeta(\tilde{x},\tilde{y})$ with\linebreak
$(\tilde{x},\tilde{y})=(\frac{1}{2}(\beta-\frac{1}{3}))^{\frac{1}{2}}(x,y)$,
one can write the steady KP equation \eqref{normalised steady KP} in the alternative form
\begin{equation}
\partial_x^2 (-\partial_x^2 \zeta + \zeta +\zeta^2)+\partial_y^2 \zeta =0, \label{Normalised KP}
\end{equation}
in which we have dropped the tildes for notational simplicity. This equation
has a family of explicit symmetric `lump' solutions of the form
\begin{equation}
\zeta_k^\star(x,y)=-6\partial_x^2 \log \tau_k^\star(x,y),\qquad k=1,2,\ldots, \label{log lump}
\end{equation}
where $\tau_k^\star$ is a polynomial of degree $k(k+1)$ with $\tau_k^\star(x,y)=\tau_k^\star(-x,y)=\tau_k^\star(x,-y)$ for all $(x,y) \in {\mathbb R}^2$;
the first two members of the family are
\begin{align*}
\tau_1^\star(x,y) &=x^2+y^2+3, \\
\tau_2^\star(x,y) &=x^6+3x^4y^2+3x^2y^4+y^6+25x^4+90x^2y^2+17y^4-125x^2+475y^2+1875.
\end{align*}
Note that the KP lump solutions \(\zeta_k^\star\) are smooth, decaying rational functions, so that the same is true of their
derivatives of all orders. The functions $\zeta_1^\star$ and $\zeta_2^\star$ are sketched in Figure \ref{lumps}.

\begin{figure}[h]
\begin{center}
\includegraphics[scale=0.33]{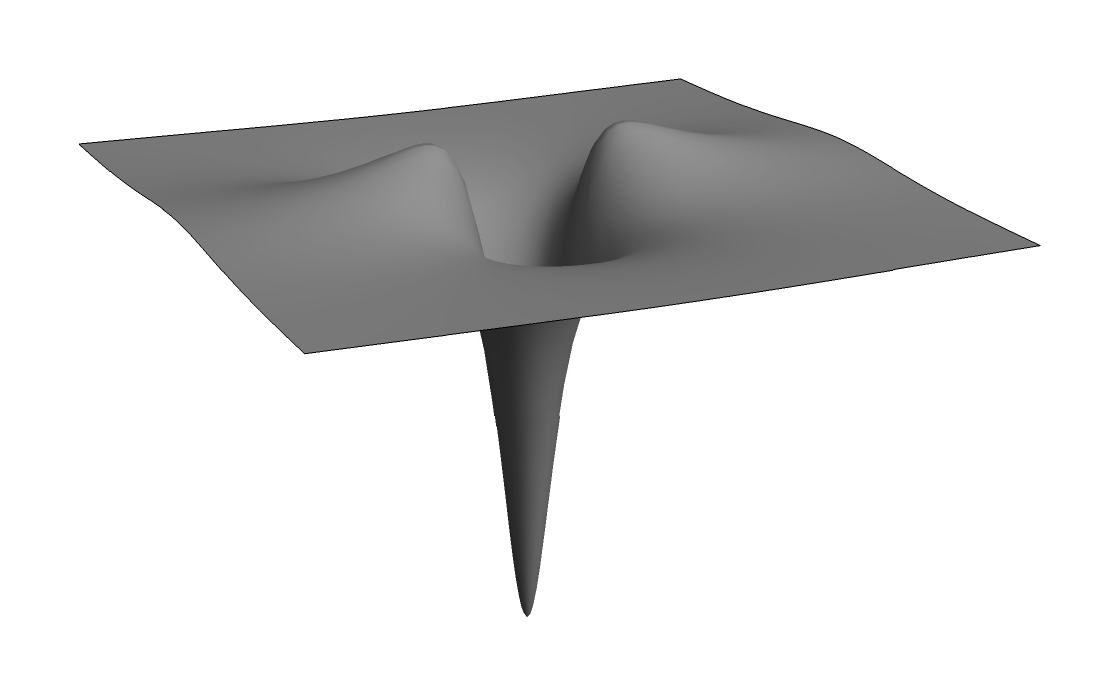}\hspace{5mm}
\includegraphics[scale=0.33]{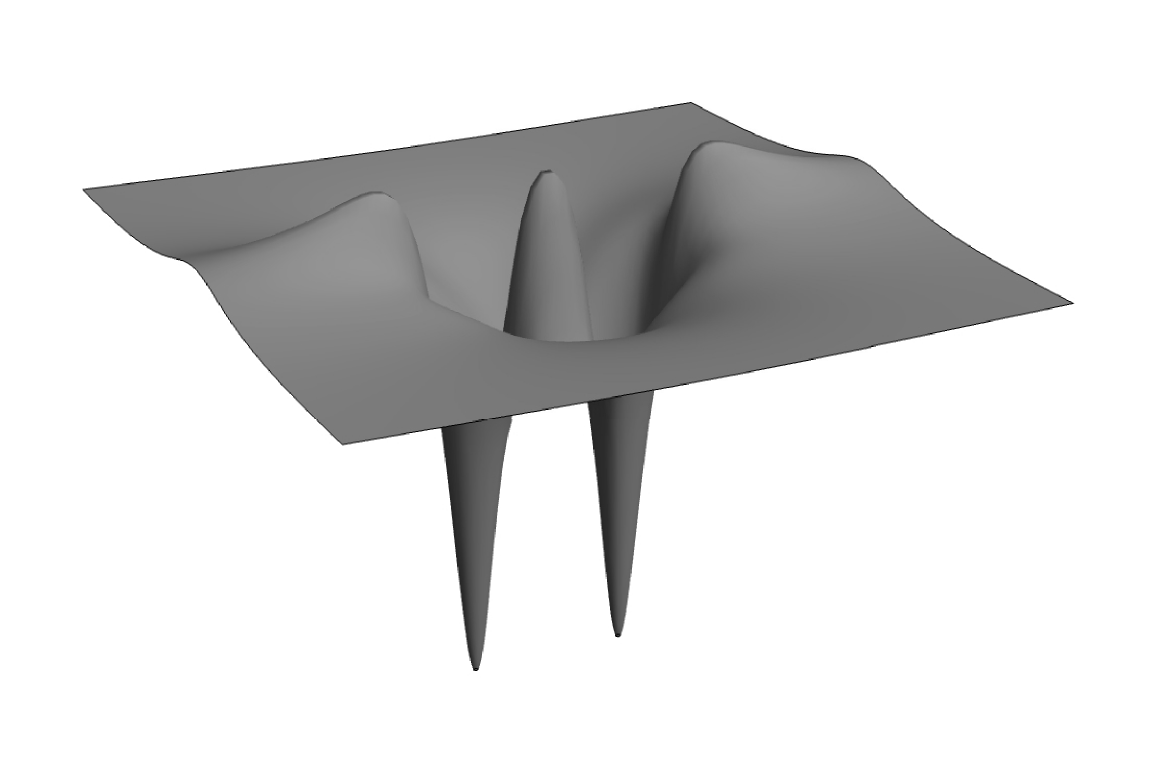}
\caption{The KP lumps $\zeta_1^\star$ (left) and $\zeta_2^\star$ (right)}\label{lumps}
\end{center}
\end{figure}

The basic lump solution $\zeta_1^\star$ was found by \citet{ManakovZakharovBordasItsMatveev77}, while the higher-order lump solutions
were discovered by \citet{PelinkovskiiStepanyants93} and fully classified by \citet{GalkinPelinovskiiStepanyants93} (see also \citet{Pelinovskii94,Pelinovskii98}, \citet{Clarkson08} and \citet{ClarksonDowie17}). These results have recently been reappraised by 
\citet{LiuWei19} and \citet{LiuWeiYang24a,LiuWeiYang24b}, who in particular discussed the nongeneracy of the lump solutions. Their work is summarised in the following result; see also the comments below the lemma.\pagebreak

\begin{lemma}$ $ \label{LWY results}

\begin{list}{(\roman{count})}{\usecounter{count}}
\item Every smooth, algebraically decaying solution of the KPI equation \eqref{Normalised KP} has the form\linebreak
$\zeta(x,y)=-6\partial_x^2 \log \tau(x,y)$,
for some polynomial $\tau$ of degree $k(k+1)$ with $k \in {\mathbb N}$ {and satisfies
 \(|\zeta(x,y)| \lesssim (1+x^2+y^2)^{-1}\) for all $(x,y) \in {\mathbb R}^2$}.
\item There is a unique symmetric solution $\zeta_k^\star$ of the form \eqref{log lump} for each $k \in {\mathbb N}$ with $k(k+1) \leq 600$.
\item The solutions $\zeta_1^\star$, $\zeta_2^\star$ are nondegenerate: the only smooth evanescent solution
of the linearised equation
\[
\partial_x^2(-\partial_x^2 w +w +2 \zeta_k^\star w)+\partial_y^2w=0
\]
for $k=1$, $2$ {that} is also invariant under $w(x,y) \mapsto w(-x,y)$ and $w(x,y) \mapsto w(x,-y)$ is $w(x,y)=0$.
\end{list}
\end{lemma}

{ It is conjectured that part (ii) actually holds for all $k \in {\mathbb N}$ \citep{LiuWeiYang24a}; furthermore a sketch of the proof of the nondegeneracy of
$\zeta_k^\star$ for $k \geq 3$ was given by \citet{LiuWeiYang24b}, and here we accept the validity of this result. 
The existence of a solitary-wave solution to the {FDKP-I} equation was proved by \citet{EhrnstroemGroves18} using a variational method, and
in this paper we considerably improve our previous result by using a perturbation argument in place of
constrained minimisation to prove the following theorem, which establishes the existence of FDKP solitary waves `close' to \(\zeta_k^\star\) for all \(k\) for which (iii) holds.}

\begin{theorem} \label{Main result}
For each $k \in {\mathbb N}$ and each sufficiently small value of $\varepsilon>0$ the FDKP-I equation \eqref{steady FDKP} posesses a smooth fully localised solitary-wave solution
$u_k^\star$ of wave speed $c=1-\varepsilon^2$, which satisfies
$$u_k^\star(x,y)=u_k^\star(-x,y)=u_k^\star(x,-y)$$
 for all $(x,y) \in {\mathbb R}^2$ and
\begin{equation}
u_k^\star(x,y)=\varepsilon^2 \zeta_k^\star(\varepsilon x,\varepsilon^2 y)+ o(\varepsilon^2) \label{uniform KP approx}
\end{equation}
uniformly over $(x,y) \in {\mathbb R}^2$.
\end{theorem}

Theorem \ref{Main result} {is proved in} Sections~\ref{Function spaces}--\ref{Existence} {below.}

\subsection{The method}

\begin{figure}[h]

\centering
\includegraphics[width=4.5cm]{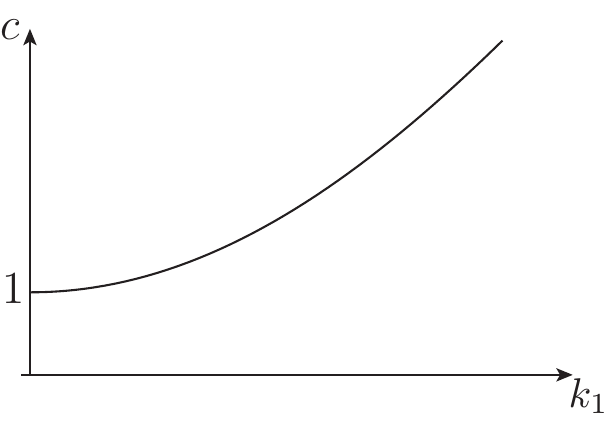}

\caption{FKDP-I dispersion relation for two-dimensional wave trains}
\label{dispersion relation}
\end{figure}

To motivate our method it is instructive to review the formal derivation of the steady KP equation
\eqref{normalised steady KP} from the steady FDKP equation \eqref{steady FDKP}.
We begin with the linear dispersion relation for {the time-dependent FDKP equation \eqref{FDKP} with $\beta>\frac{1}{3}$:} the speed $c$ and wave number $k_1$ of
a two-dimensional sinusoidal travelling wave train satisfy
\[
c= \left( 1 + \beta |k_1|^2 \right)^{\frac{1}{2}} \left( \frac{\tanh  |k_1|}{|k_1|} \right)^{\frac{1}{2}}.
\]
The function $k_1 \mapsto c(k_1)$, $k_1 \geq 0$ has a unique global minimum at $k_1=0$ with $c(0)=1$ (see Figure 
\ref{dispersion relation}).\linebreak
Bifurcations of nonlinear solitary waves are expected whenever the
linear group and phase speeds are equal, so that $c^\prime(k_1)=0$ (see \citet[\S 3]{DiasKharif99}), and
one therefore expects bifurcation of small-amplitude solitary waves from uniform flow with unit speed. 
Furthermore, observing
that $m$ is an analytic function of $k_1$ and $\frac{k_2}{k_1}$ (note that
$|\bfk|^2 = k_1^2 + \frac{k_2^2}{k_1^2}k_1^2$  for \(\bfk = (k_1,k_2)\)), one finds that
\begin{equation}
m(\bfk)=\tilde{m}(\bfk) + O(|(k_1,\tfrac{k_2}{k_1})|^4) \label{FDKP to KP}
\end{equation}
as $(k_1,\frac{k_2}{k_1}) \to 0$.
The variables \((k_1, \frac{k_2}{k_1})\) are {intrinsic to} the steady KP equation \eqref{normalised steady KP},
and corresponding to them is the scaling
\begin{equation}\label{eq:KP-scaling}
u(x,y) = \varepsilon^2\zeta(\varepsilon x, \varepsilon^2 y).
\end{equation}
{ Substituting the Ansatz \eqref{eq:KP-scaling} with assumed wave speed
$$c=1-\varepsilon^2$$
into the steady
FDKP equation \eqref{steady FDKP}, one finds that, to leading order, $\zeta$ also satisfies the normalised KP equation \eqref{normalised steady KP}.
We henceforth assume that $c=1-\varepsilon^2$ for $0 < \varepsilon < \varepsilon_0$, where
$\varepsilon_0$ is taken small enough for all our arguments to be valid.}

\begin{figure}[h]
\centering
\includegraphics[scale=0.67]{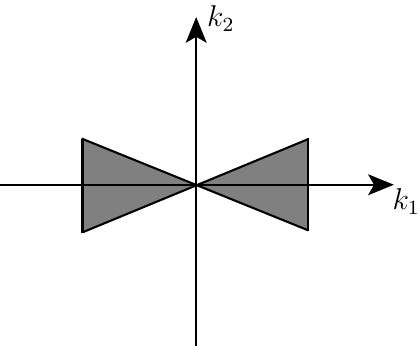}

\caption{The cone $C = \{ \bfk \in {\mathbb R}^2 \colon |k_1| \leq \delta, |\tfrac{k_2}{k_1}| \leq \delta\}$}
\label{bow tie}
\end{figure}\pagebreak

In the rigorous theory we seek solutions of \eqref{steady FDKP} in a suitable function
space $X$ and identify a corresponding phase space $Z$ for this equation. These spaces, {which are defined precisely in
Section \ref{Function spaces},} satisfy $X \subseteq Z \subseteq L^2({\mathbb R}^2)$.
The relationship \eqref{FDKP to KP} between the symbols \(m\) and \(\tilde m\) suggests that the spectrum of
a  solitary wave $u$ is concentrated in the region $|k_1|, |\frac{k_2}{k_1}| \ll 1$.
We therefore choose a fixed value of $\delta \in (0,1)$ and
decompose $L^2({\mathbb R}^2)$, and hence also $X$ and $Z$, into the direct sum
of subspaces of functions whose spectra are supported in theregion
\begin{equation}\label{eq:C}  
C=\left\{ \bfk \in {\mathbb R}^2 \colon |k_1| \leq \delta, \left|\frac{k_2}{k_1}\right| \leq \delta\right\}
\end{equation}
and its complement (see Figure \ref{bow tie}), so that
$$
X=\underbrace{\chi(\Diffb)X}_{\displaystyle = X_1} \oplus \underbrace{(1-\chi(\Diffb))X}_{\displaystyle = X_2}, \qquad
Z=\underbrace{\chi(\Diffb)Z}_{\displaystyle = Z_1} \oplus \underbrace{(1-\chi(\Diffb))Z}_{\displaystyle = Z_2},
$$
in which $\chi$ is the characteristic function of $C$. Observing that $X_1$, $Z_1$ both coincide with $\chi(\Diffb)L^2({\mathbb R}^2)$, we equip
$Z_1$ with the $L^2({\mathbb R}^2)$ norm and $X_1$ with the equivalent scaled norm
\begin{equation}\label{eq:epsilon KP norm}
|u_1|_\varepsilon^2 =  \int_{{\mathbb R}^2} \bigg(|u_1|^2 + \varepsilon^{-2} |\Diff_1 u_1|^2 + 
\varepsilon^{-2}\left|\frac{\Diff_2}{\Diff_1}u_1\right|^2\bigg)\dx\dy,
\end{equation}
and employ a method
akin to the Lyapunov--Schmidt reduction to
determine $u_2\in X_2$ as a function of $u_1\in X_1$. With $n(\Diffb)=m(\Diffb)-1$, the result is the equation
\[
\varepsilon^2 u_1 + n(\Diffb) u_1  {+} \chi(\Diffb)(u_1 + u_2(u_1))^2 =0,
\]
for \(u_1\) in the unit ball
\[
U=\{u_1 \in X_1 \colon  |u_1|_\varepsilon \leq 1\},
\] 
{ of $X_1$}.

Applying  the KP scaling
\[u_1(x,y) = \varepsilon^2\zeta(\varepsilon x, \varepsilon^2 y)\]
so that the spectrum of $\zeta$ lies in the set
\[
C_\varepsilon=\left\{ \bfk \in {\mathbb R}^2 \colon |k_1| \leq \frac{\delta}{\varepsilon}, \left|\frac{k_2}{k_1}\right|\leq \frac{\delta}{\varepsilon}\right\},
\]
one obtains the reduced equation
\begin{equation}
\varepsilon^{-2}n_\varepsilon(\Diffb)\zeta+\zeta
+ \chi_\varepsilon(\Diffb)\zeta^2 + S_\varepsilon(\zeta)=0,
\label{Final reduced eqn - Intro}
\end{equation}
where
\[
n_\varepsilon(\bfk) = n(\varepsilon k_1, \varepsilon^2 k_2), \qquad \chi_\varepsilon(\bfk) = \chi(\varepsilon k_1, \varepsilon^2 k_2).
\]
The remainder term $S_\varepsilon \colon B_M(0) \subseteq \chi_\varepsilon(\Diffb) Y^1 \rightarrow \chi_\varepsilon(\Diffb)L^2({\mathbb R}^2)$
satisfies the estimates
\[
|S_\varepsilon(\zeta)|_{L^2} \lesssim \varepsilon^2| \zeta |_{Y^1}^3, \quad
|\diff S_\varepsilon[\zeta]|_{\LL(Y^1,L^2({\mathbb R}^2))}\lesssim \varepsilon^2| \zeta |_{Y^1}^2
\]
(see Section \ref{Reduction}), where\pagebreak
\[Y^1 = \{u \in L^2({\mathbb R}^2)\colon |u|_{Y^1} < \infty\}, \qquad
|u|_{Y^1}^2 =   \int_{{\mathbb R}^2} \bigg(|u|^2 + |\Diff_1 u|^2 + 
\left|\frac{\Diff_2}{\Diff_1}u\right|^2\bigg)\dx\dy\]
is the natural energy space for the KP-I equation \citep{deBouardSaut97b};
the constant $M>1$ is chosen large enough so that $\zeta_k^\star \in B_M(0)$, while the requirement that $B_M(0)$
is contained in the range of the isomorphism $u_1 \mapsto \zeta$ requires $\varepsilon \leq M^{-2}$.
In the formal limit $\varepsilon \to 0$ the subspace $\chi_\varepsilon(\Diffb)Y^1$ `fills out' all of $Y^1$ and equation \eqref{Final reduced eqn - Intro} reduces to the
KP equation \eqref{normalised steady KP}.

We demonstrate in Theorem \ref{Final strong existence thm} that equation \eqref{Final reduced eqn - Intro} has solutions $\zeta_k^\varepsilon$
which satisfy $\zeta_k^\varepsilon \rightarrow \zeta_k^\star$ as $\varepsilon \rightarrow 0$
in a suitable subspace of $Y^1$, and deduce {our} main Theorem~\ref{Main result} by tracing back
the steps in the reduction procedure. One of the key arguments is based upon the nondegeneracy result given in Lemma \ref{LWY results}(iii),
which allows one to apply a variant of the implicit-function theorem. For this purpose we
exploit the fact that the reduction procedure preserves the invariance
of equation \eqref{steady FDKP} under $u(x,y) \mapsto u(-x,y)$ and $u(x,y) \mapsto u(x,-y)$, so that
equation \eqref{Final reduced eqn - Intro} is invariant under {$\zeta(x,y) \mapsto \zeta(-x,y)$ and $\zeta(x,y) \mapsto \zeta(x,-y)$}. 
{ It is necessary to use a low regularity version of the implicit-function theorem since the reduction in Section \ref{Reduction} is performed using the \(\varepsilon\)-dependent norm \(|\cdot|_\varepsilon\) and thus does not yield information concerning the smoothness of
$u_1$ as a function of $\varepsilon$.}

{ \citet{EhrnstroemGroves18} use a variational version of the reduction procedure outlined above to reduce a variational
principle for equation \eqref{steady FDKP} to a variational principle for \eqref{Final reduced eqn - Intro} and proceed by finding
critical points of the reduced variational functional by the direct methods of the calculus of variations. Here, with some amendments, we use their functional-analytic
setting and follow the steps in their reduction (see Sections~\ref{Function spaces} and \ref{Reduction} below), 
but study the reduced equation \eqref{Final reduced eqn - Intro} in Section \ref{Existence} in an entirely different manner,
arriving at a much more comprehensive conclusion.

Perturbation arguments to construct localised solutions
approximated by nondegenerate KP lump solutions have recently also been used for the Gross-Pitaevskii equation by \citet{LiuWangWeiYang26}
(see also \citet{ChironScheid18} for a numerical approach) and for the steady water-wave problem with strong surface tension
by \citet{GuiLaiLiuWeiYang25a} and \citet{GrovesWahlen25}, who included vorticity effects. The method has additionally been applied to
physical problems approximated by other model equations, in particular to the Whitham equation by \citet{StefanovWright20} (perturbation of Korteweg--de Vries solitary waves), to the gravity-capillary steady water-wave problem by \citet{Groves21} (perturbation of Korteweg--de Vries  and nonlinear Schr\"{o}dinger solitary waves) and \citet{BuffoniGrovesWahlen22} (perturbation of two-dimensional Schr\"{o}dinger solitary waves).

\section{Function spaces} \label{Function spaces}

In this section we introduce the Banach spaces used in our theory and state their main properties; the
proofs of most of these results are given by \citet[\S2]{EhrnstroemGroves18}. We use the familiar scale \(\{H^r({\mathbb R}^2), |\cdot|_{H^r}\}_{r \geq 0}\) of Sobolev spaces together with the anisotropic spaces
\begin{align*}
X&=\{u \in L^2({\mathbb R}^2)\colon |u|_X < \infty\}, \qquad
|u|_{X}^2 = \int_{{\mathbb R}^2} \bigg( 1 + \frac{k_2^2}{k_1^2} + \frac{k_2^4}{k_1^2} + |\bfk|^{2s}\bigg) |\hat u(\bfk)|^2\dk, \\
Z&=\{u \in L^2({\mathbb R}^2)\colon |u|_Z < \infty\}, \qquad
|u|_{Z}^2 = \int_{{\mathbb R}^2} \left( 1 + |\bfk| + k_1^2|\bfk|^{2s-3}\right) |\hat u(\bfk)|^2\dk,
\end{align*}
in which {the Sobolev index $s > \frac{3}{2}$} is fixed and $u \mapsto \hat{u}$ denotes the unitary Fourier transform on $L^2({\mathbb R}^2)$. We also use the scale $\{Y^r,|\cdot|_{Y^r}\}_{r \geq 0}$, where
\begin{equation}
Y^r = \{u \in L^2({\mathbb R}^2)\colon |u|_{Y^r} < \infty\}, \qquad
|u|_{Y^r}^2 =   \int_{{\mathbb R}^2} \bigg(1+k_1^2+\frac{k_2^2}{k_1^2}\bigg)^{\!r} |\hat u(\bfk)|^2\dk. \label{Defn of Y}
\end{equation}
Note that $Y^0=L^2({\mathbb R}^2)$ while $Y^1$ is the natural energy space for the KP-I equation
\citep{deBouardSaut97b}. \citet{EhrnstroemGroves18} use only the space $Y^1$, there called $\tilde{Y}$, but the proof of the following proposition is a straightforward variant of {the proof of} Lemma 2.1(i) in that reference.


\begin{proposition} \label{Basic embeddings}
One has the continuous embeddings
\[
Y^r \hookrightarrow L^2({\mathbb R}^2), \quad H^{s-\frac{1}{2}}({\mathbb R}^2) \hookrightarrow Z \hookrightarrow L^2({\mathbb R}^2), \quad X \hookrightarrow H^s({\mathbb R}^2)
\]
for all $ r \geq 0$, and in particular $X \hookrightarrow C_\mathrm{b}({\mathbb R}^2)$, the space of bounded, continuous functions on ${\mathbb R}^2$.
\end{proposition}

\begin{proposition} \label{Y embeddings}
The space $Y^1$ (and hence $Y^r$ for each $r \geq 1$) is
\begin{list}{(\roman{count})}{\usecounter{count}}
\item
continuously embedded in $L^p({\mathbb R}^2)$ for $2 \leq p \leq 6$,
\item
compactly embedded in $L^p_\mathrm{loc}({\mathbb R}^2)$ for $2 \leq p < 6$.
\end{list}
\end{proposition}

\begin{proposition} \label{Continuity of Y}
The space $Y^r$ is continuously embedded in $C_\mathrm{b}({\mathbb R}^2)$ for each $r> \frac{3}{2}$.
\end{proposition}
\begin{proof}
Note that
\[
|u|_\infty \lesssim \int_{{\mathbb R}^2} |\hat{u}(\bfk)|\dk \\
= \int_{{\mathbb R}^2} \bigg(1+k_1^2+\frac{k_2^2}{k_1^2}\bigg)^{\!-\frac{1}{2}r} \bigg(1+k_1^2+\frac{k_2^2}{k_1^2}\bigg)^{\!\frac{1}{2}r} |\hat{u}(\bfk)|\dk
\leq |u|_{Y^r} I^\frac{1}{2},
\]
where, with a change of variables,
\[
I=\int_{{\mathbb R}^2} \bigg(1+k_1^2+\frac{k_2^2}{k_1^2}\bigg)^{\!-r}\dk = \int_{{\mathbb R}^2} (1+|\bfk|^2)^{-r} |k_1|\dk < \infty.
\]
The continuity of $u$ follows from a standard dominated convergence argument with $\hat{u}$ as dominating function.
\end{proof}

\begin{proposition} $ $ \label{X to Z mappings}
\begin{list}{(\roman{count})}{\usecounter{count}}
\item
The Fourier multiplier $m(\Diffb)$ maps $X$ continuously onto $Z$.
\item
The formula $u \mapsto u^2$ maps $X$ smoothly into $Z$.
\end{list}
\end{proposition}

We decompose $u \in L^2({\mathbb R}^2)$ into the sum of functions $u_1$ and $u_2$ whose spectra are supported in the region \(C\) {defined in} \eqref{eq:C} and its complement (see Figure \ref{bow tie}) by writing
\[
u_1 = \chi(\Diffb)u, \qquad u_2=(1-\chi(\Diffb))u,
\]
where $\chi$ is the characteristic function of $C$. Since they are subspaces of $L^2({\mathbb R}^2)$, the Fourier multiplier $\chi(\Diffb)$ induces the orthogonal decomposition \(X = X_1 \oplus X_2\)
with \(X_1{ = } \chi(\Diffb) X\), \(X_2 {= } (1 - \chi(\Diffb)) X\) and analogous decompositions of the spaces $Y^r$ and $Z$. We write $Z=Z_1 \oplus Z_2$, but retain the explicit notation $\chi(\Diffb)Y^r$ and $\chi(\Diffb)L^2({\mathbb R}^2)$. The spaces $X_1$, $Z_1$ and $\chi(\Diffb)Y^r$ all coincide with $\chi(\Diffb)L^2({\mathbb R}^2)$, and
$|\cdot|_{L^2}$, $|\cdot|_{X}$, $|\cdot|_{Y^r}$ and $|\cdot|_{Z}$ are all equivalent
norms for these spaces. We do however make specific choices in the theory in below; we equip
$Z_1$ and $\chi(\Diffb)Y^r$ with $|\cdot|_{L^2}$ and $|\cdot|_{Y^r}$ respectively, and $X_1$ with the
equivalent scaled norm
$$
|u_1|_\varepsilon^2 = \int_{{\mathbb R}^2} \bigg( 1 + \varepsilon^{-2} k_1^2 + \varepsilon^{-2} \frac{k_2^2}{k_1^2} \bigg) |\hat u_1(\bfk)|^2\dk
$$
(see equation \eqref{eq:epsilon KP norm}) in anticipation of the KP scaling $(k_1, k_2) \mapsto (\varepsilon k_1, \varepsilon^2 k_2)$.\pagebreak

\begin{proposition} \label{n is an iso}
The mapping $n(\Diffb)=m(\Diffb)-1$ is an isomorphism $X_2 \to Z_2$.
\end{proposition}

\begin{proposition}\label{Products}
The estimates
\[|\partial_x^{m_1}\partial_y^{m_2} u_1|_\infty \lesssim \varepsilon |u_1|_\varepsilon, \qquad m_1,m_2 \geq 0,\]
and
\[
| u_1 v |_Z \lesssim \varepsilon | u_1 |_\varepsilon |v|_{X}, \qquad |v w|_Z \lesssim |v|_X |w|_X
\] 
hold for all $u_1 \in X_1$ and $v,w \in X$. 
\end{proposition}

Finally, we introduce the space $Y_\varepsilon^r=\chi_\varepsilon(\Diffb)Y^r$, where
$\chi_\varepsilon(k_1,k_2)=\chi(\varepsilon k_1, \varepsilon^2 k_2)$ (with norm
$|\cdot|_{Y^r}$), noting the relationship
\[|u|_\varepsilon^2=\varepsilon|\zeta|_{Y^1}^2,
\qquad  u(x,y)=\varepsilon^2 \zeta(\varepsilon x, \varepsilon^2 y)
\]
for $\zeta \in Y_\varepsilon^1$. Observe that $Y_\varepsilon^r$ coincides with
$\chi_\varepsilon(\Diffb)X$, $\chi_\varepsilon(\Diffb)Z$ and $\chi_\varepsilon(\Diffb)L^2({\mathbb R}^2)$
for $\varepsilon>0$, and with $Y^r$ in the limit $\varepsilon \rightarrow 0$.

\section{Reduction} \label{Reduction}

We proceed by making the Ansatz $c = 1 - \varepsilon^2$ and studying equation \eqref{steady FDKP}
in its phase space $Z$.
Note that $u = u_1 + u_2 \in X_1 \oplus X_2$ satisfies this equation if and only if
\begin{align}
\varepsilon^2 u_1 + n(\Diffb) u_1  {+} \chi(\Diffb)(u_1 + u_2)^2 =0, & \qquad \mbox{ in } Z_1,\label{system i} \\
\varepsilon^2 u_2 + n(\Diffb) u_2  {+} (1-\chi(\Diffb))(u_1 + u_2)^2 =0, & \qquad \mbox{ in } Z_2. \label{system ii}
\end{align}
The first step is to solve \eqref{system ii} for $u_2$ as a function of $u_1$ using the following
result, which is proved by a straightforward application of the contraction mapping principle. 

\begin{theorem} \label{fixed-point}
Let $\WW_1$, $\WW_2$ be Banach spaces, $K$ be a continuous function $\overline{B}_1(0) \subseteq \WW_1 \to [0,\infty)$
and $\FF \colon \overline{B}_1(0) \times \WW_2 \to \WW_2$ be a smooth function satisfying
\[
|\FF(w_1,0)|_{\WW_2} \leq \tfrac{1}{2} K(w_1), \qquad |\diff_2 \FF[w_1,w_2]|_{\LL(\WW_2,\WW_2)} \leq \tfrac{1}{3}
\]
for all $(w_1,w_2) \in \overline{B}_1(0) \times \overline{B}_{K(w_1)}(0)$. The fixed-point equation
\[
w_2 = \FF(w_1, w_2)
\]
has for each $w_1 \in \overline{B}_1(0)$ a unique solution $w_2 = w_2(w_1) \in \overline{B}_{K(w_1)}(0)$. Moreover
$w_2$ is a smooth function of $w_1$ and
satisfies
\[
|\diff w_2[w_1] |_{\LL(\WW_1,\WW_2)} \lesssim | \diff_1 \FF[w_1,w_2(w_1)] |_{\LL(\WW_1,\WW_2)}.
\]
\end{theorem}

Write \eqref{system ii} as 
\begin{equation}\label{eqn with G}
u_2=\FF(u_1,u_2), 
\end{equation}
where
\begin{equation}
\label{G}
\FF(u_1,u_2) = -n(\Diffb)^{-1} {(1-\chi(\Diffb) ) \left( \varepsilon^2 u_2 +  (u_1 + u_2)^2 \right)};
\end{equation}
the following mapping property of $\FF$ follows from Propositions \ref{X to Z mappings}(ii) and \ref{n is an iso}.

\begin{proposition} \label{G weak}
Equation \eqref{G} defines a smooth mapping $\FF\colon X_1 \times X_2 \to X_2$.
\end{proposition}

\begin{lemma}\label{u_2}
{  Define $U = \{u_1 \in X_1\colon |u_1|_\varepsilon \leq 1\}$.}
Equation \eqref{eqn with G} defines a map
\[U \ni u_1 \mapsto  u_2(u_1) \in X_2\]
which satisfies
\[|u_2(u_1)|_{X_2} \lesssim \varepsilon |u_1|_\varepsilon^2, \qquad |\diff u_2[u_1]|_{\LL(X_1,X_2)} \lesssim \varepsilon |u_1|_\varepsilon.\]
\end{lemma}
\begin{proof}
We apply Theorem~\ref{fixed-point} to equation \eqref{eqn with G} with $\WW_1=(X_1,|\cdot|_\varepsilon)$,
$\WW_2=(X_2,|\cdot|_X)$. Note that
\begin{eqnarray*}
\diff_1 \FF[u_1,u_2](v_1) &= & -n(\Diffb)^{-1}{(1-\chi(\Diffb) )(2(u_1+u_2)v_1)}, \nonumber \\
\diff_2 \FF[u_1,u_2](v_2) &= &- n(\Diffb)^{-1}{(1-\chi(\Diffb))(\varepsilon^2 v_2 + 2(u_1+u_2)v_2)}
\end{eqnarray*}
and
\[
| (n(\Diffb))^{-1} {(1-\chi(\Diffb))} u|_X \lesssim |u|_Z,
\]
by Proposition \ref{n is an iso}. Using Proposition \ref{Products}, we therefore find that
\[
|\FF(u_1,0)|_X \lesssim |u_1^2|_Z \lesssim   \varepsilon  |u_1|_\varepsilon |u_1|_X
\lesssim \varepsilon |u_1|_\varepsilon |u_1|_{L^2} \leq \varepsilon |u_1|_\varepsilon^2
\]
and
\begin{eqnarray*}
|\diff_2  \FF[u_1,u_2](v_2)|_X
&\lesssim \varepsilon^2 |v_2|_Z + |u_1 v_2|_Z  + |u_2 v_2|_Z\nonumber \\
&\lesssim (\varepsilon^2 + \varepsilon |u_1|_\varepsilon + |u_2|_X) |v_2|_{X}.
\end{eqnarray*}
To satisfy the assumptions of Theorem~\ref{fixed-point},
we choose $K(u_1)=\sigma \varepsilon |u_1|_\varepsilon^2$ for a sufficiently large value
of $\sigma>0$, so that
\[
|u_2|_X \lesssim \tfrac{1}{2}K(u_1), \qquad
 |\diff_2  \FF[u_1,u_2]|_{\LL(X_2,X_2)}  \lesssim \varepsilon
\]
for $(u_1,u_2)  \in U \times  \overline{B}_{K(u_1)}(0)$. The theorem asserts the existence
of a unique solution $u_2(u_1) \in \overline{B}_{K(u_1)}(0)$
of \eqref{eqn with G} for each $u_1 \in U$
which satisfies
\[
|u_2(u_1)|_X \lesssim \varepsilon | u_1 |_\varepsilon^2
\] and
\begin{align*}
|\diff u_2[u_1](v_1) |_X & \lesssim | \diff_1 \FF[u_1,u_2(u_1)] (v_1)|_X \\
&\lesssim |u_1 v_1|_Z + |u_2(u_1) v_1|_Z\nonumber \\  
&\lesssim \varepsilon  (|u_1|_X + |u_2(u_1)|_X) |v_1|_\varepsilon\nonumber \\
&\lesssim \varepsilon (|u_1|_\varepsilon + \varepsilon |u_1|_\varepsilon^2) |v_1|_\varepsilon,
\end{align*}
where we have used Proposition \ref{Products}.
\end{proof}

Our next result shows in particular that $u=u_1+u_2(u_1)$ belongs to $H^\infty({\mathbb R}^2)=\smash{\bigcap\limits_{j=1}^\infty} H^j({\mathbb R}^2)$
for each $u_1 \in U_1$.

\begin{proposition} \label{u is smooth}
Any function $u=u_1+u_2 \in X_1 \oplus X_2$ which satisfies \eqref{eqn with G}
belongs to $H^\infty({\mathbb R}^2)$.
\end{proposition}
\begin{proof} Obviously $u_1 \in H^\infty({\mathbb R}^2)$, and to show that $u_2$ is also smooth we {abandon} the fixed regularity index in the spaces $X$ and $Z$ and state it explicitly as a variable parameter.
Since $H^{s}({\mathbb R}^2)$ is an algebra for $s > \frac{3}{2}$ and $X^s \hookrightarrow
(1-\chi(D))H^s({\mathbb R}^2) \hookrightarrow Z_2^{s+\frac{1}{2}}$ (see Proposition \ref{Basic embeddings}),
the mapping\pagebreak
\[
X_1 \oplus X_2^{s} \ni (u_1, u_2) \mapsto -(1-\chi(\Diffb) ) \left( \varepsilon^2 u_2 +  (u_1 + u_2)^2 \right) \in Z_2^{s+\frac{1}{2}}
\]
is continuous. It follows that $u_2 \in X_2^{s+\frac{1}{2}}$, because $n(\Diffb)$ is an isomorphism $X_2^{s+\frac{1}{2}} \to Z_2^{s+\frac{1}{2}}$ by Proposition~\ref{n is an iso}. Bootstrapping this argument yields
$u_2 \in X_2^s \subset H^s({\mathbb R}^2)$ for any $s \in {\mathbb R}$.
\end{proof}

The next step is to substitute $u_2=u_2(u_1)$ into equation \eqref{system i} to obtain the reduced equation
\[
\varepsilon^2 u_1 + n(\Diffb) u_1  {+} \chi(\Diffb)(u_1 + u_2(u_1))^2 =0
\]
for $u_1$. We can write this equation as
\[
\varepsilon^2 u_1 + n(\Diffb) u_1  {+} \chi(\Diffb)u_1^2 + R_\varepsilon(u_1)=0,
\]
where
\begin{equation}
R_\varepsilon(u_1)=\chi(\Diffb) \big(2u_1u_2(u_1) + u_2(u_1)^2\big). \label{Defn of Re}
\end{equation}

\begin{proposition}
The function $R_\varepsilon\colon U \subseteq X_1 \rightarrow Z_1$ satisfies the estimates
\[|R_\varepsilon(u_1)|_{L^2} \lesssim \varepsilon^2 |u_1|_\varepsilon^3, \qquad |\diff R_\varepsilon[u_1]|_{\LL(X_1,L^2({\mathbb R}^2))} \lesssim \varepsilon^2 |u_1|_\varepsilon^2.\]
\end{proposition}
\begin{proof}
By using Proposition \ref{Products} and Lemma \ref{u_2} it follows from \eqref{Defn of Re} that
\begin{align*}
|R_\varepsilon(u_1)|_{L^2} &\lesssim |u_1 u_2(u_1)|_Z + |u_2(u_1)^2|_Z \\
& \lesssim  \varepsilon |u_1|_\varepsilon |u_2(u_1)|_X + |u_2(u_1)|_X^2 \\
& \lesssim \varepsilon^2 |u_1|_\varepsilon^3,
\end{align*}
and from
\[
\diff R_\varepsilon[u_1](v_1) = \chi(\Diffb) \big(2 v_1 u_2(u_1) +u_1 \diff u_2[u_1](v_1) + 2 u_2(u_1)\diff u_2[u_1](v_1)\big)
\]
that
\begin{align*}
|\diff R_\varepsilon[u_1](v_1) |_{L^2} &\lesssim |v_1 u_2(u_1)|_Z + |u_1 \diff u_2[u_1](v_1) |_Z + |u_2(u_1)\diff u_2[u_1](v_1)|_Z \\
& \lesssim  \varepsilon |v_1|_\varepsilon |u_2(u_1)|_X + \varepsilon |u_1|_\varepsilon |\diff u_2[u_1](v_1))|_X + |u_2(u_1)|_X|\diff u_2[u_1](v_1)|_X\\
& \lesssim \varepsilon^2 |u_1|_\varepsilon^2 |v_1|_\varepsilon.\qedhere
\end{align*}
\end{proof}

The reduction is completed by introducing the KP scaling
\[
u_1(x,y) = \varepsilon^2 \zeta(\varepsilon x,\varepsilon^2 y),
\]
noting that $I\colon u_1 \mapsto \zeta$ is an isomorphism $X_1 \rightarrow Y_\varepsilon^1$ and $Z_1 \rightarrow Y_\varepsilon^0$
and choosing $M>1$ large enough so that $\zeta_k^\star \in B_M(0)$ (and $\varepsilon<M^{-2}$, so that
$B_M(0) \subseteq Y_\varepsilon^1$  is contained in $I[U]=B_{\varepsilon^{\smash{-\frac{1}{2}}}}(0)\subseteq Y_\varepsilon^1$).
Here we have replaced $(Z_1,|\cdot|_{L^2})$
by the identical space $(Y_\varepsilon^0, |\cdot|_{Y^0})$ in order to work exclusively with the scales $\{Y^r,|\cdot|_{Y^r}\}_{r \geq 0}$ and
$\{Y_\varepsilon^r,|\cdot|_{Y^r}\}_{r \geq 0}$ of function spaces.
We find that $\zeta \in B_M(0) \subseteq Y_\varepsilon^1$ satisfies the equation
\begin{equation}
\varepsilon^{-2}n_\varepsilon(\Diffb)\zeta+\zeta
+ \chi_\varepsilon(\Diffb)\zeta^2 + S_\varepsilon(\zeta)=0,
\label{Final reduced eqn}
\end{equation}
which now holds in $Y_\varepsilon^0$, where
\[n_\varepsilon(\bfk) = n(\varepsilon k_1, \varepsilon^2 k_2)\]
and $S_\varepsilon\colon B_M(0) \subseteq Y_\varepsilon^1 \rightarrow Y_\varepsilon^0$ satisfies the estimates
\begin{equation}
|S_\varepsilon(\zeta)|_{Y^0} \lesssim \varepsilon | \zeta |_{Y^1}^3, \quad
|\diff S_\varepsilon[\zeta]|_{\LL(Y^1,Y^0)}\lesssim \varepsilon| \zeta |_{Y^1}^2. \label{S estimates}
\end{equation}
Note that $|u_1|_\varepsilon^2=\varepsilon |\zeta|_{Y^1}^2$ and that the change of variables
from $(x,y)$ to $(\varepsilon x, \varepsilon^2y)$ introduces a further factor of $\varepsilon^{\frac{3}{2}}$ in the
remainder term.

Finally, observe that the FDKP equation
\[
\varepsilon^2 u + n(\Diffb) u  +u^2 =0
\]
is invariant under $u(x,y) \mapsto u(-x,y)$ and $u(x,y) \mapsto u(x,-y)$ and the reduction procedure
preserves this invariance: equation \eqref{Final reduced eqn} is invariant under $\zeta(x,y) \mapsto \zeta(-x,y)$ and $\zeta(x,y) \mapsto \zeta(x,-y)$.

\medskip

\section{Solution of the reduced equation} \label{Existence}

In this section we construct solitary-wave solutions of the reduced equation \eqref{Final reduced eqn},
noting that in the formal limit $\varepsilon \rightarrow 0$ it reduces to the KP equation \eqref{normalised steady KP},
which has explicit (symmetric) solitary-wave solutions $\zeta_k^\star$.
For this purpose we use a  perturbation argument, rewriting \eqref{Final reduced eqn}
as a fixed-point equation and applying the following variant of the implicit-function theorem. 
It is necessary to use a low regularity version of the implicit-function theorem since the reduction in Section \ref{Reduction}
is performed using the \(\varepsilon\)-dependent norm \(|\cdot|_\varepsilon\) and thus does not yield information concerning the smoothness of
$u_1$ as a function of $\varepsilon$.

\begin{theorem} \label{IFT}
Let $\WW$ be a Banach space, $W_0$ and $\Lambda_0$ be open neighbourhoods of respectively $w^\star$ in $\WW$ and the origin in ${\mathbb R}$, and $\GG\colon  W_0 \times \Lambda_0 \rightarrow \WW$ be a function which is differentiable with respect to $w \in W_0$ for each $\lambda \in \Lambda_0$.
Furthermore, suppose that 
$\GG(w^\star,0)=0$, $\diff_1\GG[w^\star,0]\colon \WW \rightarrow \WW$ is an isomorphism,
\[
\lim_{w \rightarrow w^\star}|\diff_1\GG[w, 0]-\diff_1\GG[w^\star,0]|_{\LL(\WW,\WW)}=0
\]
and
\[\lim_{\lambda \rightarrow 0} |\GG(w,\lambda)-\GG(w,0)|_{\WW}=0, \quad \lim_{\lambda \rightarrow 0} \
|\diff_1\GG[w,\lambda]-\diff_1\GG[w,0]|_{\LL(\WW,\WW)}=0\]
uniformly over $w \in W_0$.

There exist open neighbourhoods $W \subseteq W_0$ of $w^\star$ in $\WW$ and $\Lambda \subseteq \Lambda_0$ of the origin in ${\mathbb R}$, and a uniquely determined mapping
$h\colon \Lambda \rightarrow W$ with the properties that
\begin{list}{(\roman{count})}{\usecounter{count}}
\item
$h$ is continuous at the origin with $h(0)=w^\star$,
\item
$\GG(h(\lambda),\lambda)=0$ for all $\lambda \in \Lambda$,
\item
$w=h(\lambda)$ whenever $(w,\lambda) \in  W \times \Lambda$ satisfies $\GG(w,\lambda)=0$.
\end{list}
\end{theorem}

Our main result is the following theorem, which is proved by reformulating equation \eqref{Final reduced eqn}
in an appropriately chosen function space and verifying that it satisfies the assumptions of Theorem~\ref{IFT} 
through a series of auxiliary results.

\begin{theorem} \label{Final strong existence thm}
Fix $\theta \in (\frac{1}{2},1)$.
For each sufficiently small value of $\varepsilon>0$ equation \eqref{Final reduced eqn} has a small-amplitude, symmetric
solution $\zeta_k^\varepsilon$ in $Y_\varepsilon^{1+\theta}$
with $|\zeta_k^\varepsilon-\zeta_k^\star|_{Y^{1+\theta}} \rightarrow 0$ as $\varepsilon \rightarrow 0$.
\end{theorem}

The first step in the proof of Theorem \ref{Final strong existence thm} is to write \eqref{Final reduced eqn} as the fixed-point equation
\begin{equation}
\zeta+\varepsilon^2\big(n_\varepsilon(\Diffb)+\varepsilon^2\big)^{-1}
\big(\chi_\varepsilon(\Diffb)\zeta^2+S_\varepsilon(\zeta)\big)=0 \label{Final reduced eqn FP}
\end{equation}
and use the following results to `replace' $\varepsilon^2\left(n_\varepsilon(\Diffb)+\varepsilon^2\right)^{-1}$ with $\tilde{m}(\Diffb)^{-1}$.\pagebreak

\begin{proposition} \label{Replace nonloc with diff}
The inequality
\[\left|\frac{\varepsilon^2}{\varepsilon^2 + n_\varepsilon(\bfk)}- \frac{1}{\tilde{m}(\bfk)}\right| \lesssim \frac{\varepsilon}{\big(1+|(k_1,\frac{k_2}{k_1})|^2\big)^{1/2}}\]
holds uniformly over $|k_1|, |\tfrac{k_2}{k_1}| < \frac{\delta}{\varepsilon}$.
\end{proposition}
\begin{proof}
Recall that \(\beta > \frac{1}{3}\). Clearly
\[
\left|\frac{\varepsilon^2}{\varepsilon^2 + n_\varepsilon(\bfk)}- \frac{1}{\tilde{m}(\bfk)}\right|
=\frac{\left|n_\varepsilon(\bfk)-(\beta-\tfrac{1}{3})\varepsilon^2 k_1^2-\varepsilon^2\tfrac{k_2^2}{k_1^2}\right|}{(\varepsilon^2 + n_\varepsilon(\bfk))
\left(1+(\beta-\tfrac{1}{3})k_1^2+\tfrac{k_2^2}{k_1^2}\right)}.
\]
Furthermore, since $n(s)$ is an analytic function of $s_1$ and $\tfrac{s_2}{s_1}$, we have that
\[
\left|n(s)-\left(\beta-\frac{1}{3}\right)s_1^2-\frac{s_2^2}{s_1^2}\right|\lesssim \left|\left(s_1,\frac{s_2}{s_1}\right)\right|^3\]
and by the definition of \(n\) that
\[
n(s) \gtrsim \left|\left(s_1,\frac{s_2}{s_1}\right)\right|^2\]
for $|(s_1,\tfrac{s_2}{s_1})| \leq \delta$.
It follows that
\[
\left|\frac{\varepsilon^2}{\varepsilon^2 + n_\varepsilon(\bfk)}- \frac{1}{\tilde{m}(\bfk)}\right|
\lesssim
\frac{\varepsilon |(k_1,\frac{k_2}{k_1})|^3}{(1+|(k_1,\frac{k_2}{k_1})|^2)^2}
\]
uniformly over $|k_1|, |\tfrac{k_2}{k_1}| < \frac{\delta}{\varepsilon}$.
\end{proof}

\begin{corollary} \label{Replacement}
For each $\theta \in [0,1]$ the inequality
\[\left|\frac{\varepsilon^2}{\varepsilon^2 + n_\varepsilon(\bfk)}- \frac{1}{\tilde{m}(\bfk)}\right| \lesssim \frac{\varepsilon^{1-\theta}}{(1+|(k_1,\frac{k_2}{k_1})|^2)^{\frac{1}{2}(1+\theta)}}\]
holds uniformly over $|k_1|, |\tfrac{k_2}{k_|}| < \frac{\delta}{\varepsilon}$.
\end{corollary}
\begin{proof}
This result follows from Proposition \ref{Replace nonloc with diff} and the observation that
$\varepsilon \lesssim \delta(1+|(k_1,\frac{k_2}{k_2})|^2)^{-\frac{1}{2}}$ for $|k_1|, |\tfrac{k_2}{k_1}| < \frac{\delta}{\varepsilon}$.
\end{proof}

Using Corollary \ref{Replacement}, one can write equation \eqref{Final reduced eqn FP} as
\begin{equation}
\zeta+F_\varepsilon(\zeta)=0, \label{Abstract FP}
\end{equation}
in which
\[
F_\varepsilon(\zeta)= \tilde{m}(\Diffb)^{-1}\chi_\varepsilon(\Diffb)\zeta^2
+\underbrace{T_{1,\varepsilon}(\zeta)+T_{2,\varepsilon}(\zeta)}_{\displaystyle = T_\varepsilon(\zeta)}
\]
and
\[
T_{1,\varepsilon}(\zeta) = \left(\varepsilon^2\big(n_\varepsilon(\Diffb)+\varepsilon^2\big)^{-1}-\tilde{m}(\Diffb)^{-1}\right)\chi_\varepsilon(\Diffb)\zeta^2,
\qquad
T_{2,\varepsilon}(\zeta) =
\varepsilon^2\big(n_\varepsilon(\Diffb)+\varepsilon^2\big)^{-1}S_\varepsilon(\zeta).
\]

\begin{proposition}\label{prop:T}
Fix $\theta \in [0,1]$. The mapping
$T_\varepsilon\colon B_M(0) \subseteq Y_\varepsilon^1 \rightarrow Y_\varepsilon^{1+\theta}$ satisfies
\[
|T_\varepsilon(\zeta)|_{Y^{1+\theta}} \lesssim \varepsilon^{1-\theta} | \zeta |_{Y^1}^2, \quad
|\diff T_\varepsilon[\zeta]|_{\LL(Y^1,Y^{1+\theta})}\lesssim \varepsilon^{1-\theta}| \zeta |_{Y^1}
\]
for all $\zeta \in Y_\varepsilon^{1+\theta}$.
\end{proposition}
\begin{proof}
The result for $T_{1,\varepsilon}$ follows from the calculation
$$\left|\left(\varepsilon^2\big(n_\varepsilon(\Diffb)+\varepsilon^2\big)^{-1}\!\!-\tilde{m}(\Diffb)^{-1}\right)\!\chi_\varepsilon(\Diffb)\zeta\rho\right|_{Y^{1+\theta}}
\!\!\lesssim\! \varepsilon^{1-\theta}|\zeta\rho|_0 \!\lesssim\! \varepsilon^{1-\theta}|\zeta|_{L^4} |\rho|_{L^4} \!\lesssim\! \varepsilon^{1-\theta}|\zeta|_{Y^{1+\theta}}|\rho|_{Y^{1+\theta}}$$
for all $\zeta$, $\rho \in Y_\varepsilon^{1+\theta}$ (see Corollary \ref{Replacement} and Proposition \ref{Y embeddings}(i)).
Corollary \ref{Replacement} (with $\theta=1$) also yields
$$
\frac{\varepsilon^2}{n_\varepsilon(\bfk)+\varepsilon^2} \lesssim \left(1+k_1^2+\frac{k_2^2}{k_1^2}\right)^{-1},
$$
and the result for $T_{2,\varepsilon}$ follows from this estimate and \eqref{S estimates}.
\end{proof}

\begin{remark}
We can also consider $T_\varepsilon$ as a mapping $T_\varepsilon\colon B_M(0) \subseteq Y_\varepsilon^{1+\theta} \rightarrow Y_\varepsilon^{1+\theta}$ with identical estimates since $\{Y_\varepsilon^r,|\cdot|_{Y^r}\}_{r \geq 0}$ is a scale of Banach spaces.
\end{remark}

It is convenient to replace equation \eqref{Abstract FP} with
\[
\zeta+G_\varepsilon(\zeta)=0,
\]
where $G_\varepsilon(\zeta) = F_\varepsilon(\chi_\varepsilon(\Diffb)\zeta)$, and study it in the fixed space $Y^{1+\theta}$
for $\theta\in(\frac{1}{2},1)$
(the solution sets of the two equations evidently coincide); we choose $\theta>\frac{1}{2}$ so that $Y^{1+\theta}$ is embedded
in $C_\mathrm{b}({\mathbb R}^2)$ and $\theta<1$ so that $T_\varepsilon(\zeta)$ vanishes in the limit $\varepsilon \rightarrow 0$. Note that the regularity index \(s\) for the space \(X\) must be taken larger than \(r = 1+\theta\) to preserve the embedding \(X \hookrightarrow Y^r\) (see Lemma~\ref{Basic embeddings});\
in fact all desired properties are satisfed for \(\tfrac{3}{2} < 1 + \theta < s < 2\). We establish Theorem \ref{Final strong existence thm} by applying Theorem \ref{IFT} with
\begin{equation}\label{eq:XX}
\WW=Y^{1+\theta}_\mathrm{e}= \{\zeta \in Y^{1+\theta}\colon \mbox{$\zeta(x,y)=\zeta(-x,y)=\zeta(x,-y)$ for all $(x,y) \in {\mathbb R}^2$}\},
\end{equation}
$W_0=B_M(0)\subseteq Y_\mathrm{e}^{1+\theta}$, $\Lambda_0=(-\varepsilon_0,\varepsilon_0)$ for a sufficiently small value of $\varepsilon_0$,
and
\begin{equation}\label{eq:GG}
\GG(\zeta,\varepsilon)=\zeta+G_{|\varepsilon|}(\zeta),
\end{equation}
where $\varepsilon$ has been replaced by $|\varepsilon|$ to have  $\GG(\zeta,\varepsilon)$ defined for $\varepsilon$ in
a full neighbourhood of the origin in ${\mathbb R}$.

We begin by verifying that the functions \(\zeta_k^\star\) belong to $Y_\mathrm{e}^{1+\theta}$.

\begin{proposition}
Each KP lump solution \(\zeta_k^\star\) belongs to $Y^2$.
\end{proposition}
\begin{proof}
First note that $(\zeta_k^\star)^2$ belongs to $L^2({\mathbb R}^2)=Y^0$ because \(|\zeta_k^\star(x,y)| \lesssim (1+x^2+y^2)^{-1}\) for all $(x,y) \in {\mathbb R}^2$
(see Proposition \ref{LWY results}(i)). Since \(\zeta_k^\star\) satisfies
$$
 \zeta_k^\star+\tilde{m}(\Diffb)^{-1}(\zeta_k^\star)^2=0
$$
and \(\tilde m(\Diffb)^{-1}\) is a lifting operator of order \(2\) for the scale $\{Y^r,|\cdot|_{Y^r}\}_{r \geq 0}$, one finds that \(\zeta_k^\star \in Y^2\). 
\end{proof}

Observe that \(\GG(\cdot,\varepsilon)\) is a continuously differentiable function \(B_M(0) \subseteq Y_\mathrm{e}^{1+\theta} \to Y_\mathrm{e}^{1+\theta}\)
for each fixed $\varepsilon \geq 0$, so that
\[
\lim_{\zeta \rightarrow \zeta_k^\star}|\diff_1\GG[\zeta, 0]-\diff_1\GG[\zeta_k^\star,0]|_{\LL(Y^{1+\theta},Y^{1+\theta})}=0.
\]
The facts that
$$
 \lim_{\varepsilon \rightarrow 0} |\GG(\zeta,\varepsilon)-\GG(\zeta,0)|_{Y^{1+\theta}}=0, \quad \lim_{\varepsilon \rightarrow 0} \
|\diff_1\GG[\zeta,\varepsilon]-\diff_1\GG[\zeta,0]|_{\LL(Y^{1+\theta},Y^{1+\theta})}=0
$$
uniformly over $\zeta \in B_M(0)\subseteq Y_\mathrm{e}^{1+\theta}$ are obtained from the equation
\[
\GG(\zeta,\varepsilon)-\GG(\zeta,0)
= \tilde{m}(\Diffb)^{-1}
\left(\chi_\varepsilon(\Diffb)
\left(\chi_\varepsilon(\Diffb)\zeta\right)^2-\zeta^2\right)
 +  T_{|\varepsilon|}(\zeta)
 \]
using Proposition~\ref{prop:T} and Corollary~\ref{Difference of G 3} below, which is a consequence of the next two lemmas.

\begin{lemma} \label{Difference of G 1}
Fix $\theta > \frac{1}{2}$. The estimate
\[
| \tilde m(\Diffb)^{-1} \chi_\varepsilon(\Diffb) \big( (( \chi_\varepsilon(\Diffb) + I) \zeta) ((\chi_\varepsilon(\Diffb) - I) \rho)\big) |_{Y^{1+\theta}} \lesssim \varepsilon | \zeta |_{Y^{1+\theta}} | \rho |_{Y^{1+\theta}}
\]
holds for all \(\zeta, \rho \in Y^{1+\theta}\).
\end{lemma}\enlargethispage{5mm}
\begin{proof}
Recall that \(\tilde m(\Diffb)^{-1}\) is a lifting operator of order 2 for the scale $\{Y^r,|\cdot|_{Y^r}\}_{r \geq 0}$ and that \(\chi_\varepsilon(\Diffb)\) is a bounded projection on $L^2({\mathbb R}^2)$. It follows that
\begin{align*}
| \tilde m(\Diffb)^{-1} & \chi_\varepsilon(\Diffb) \big( (( \chi_\varepsilon(\Diffb) + I) \zeta) ((\chi_\varepsilon(\Diffb) - I) \rho)\big) |_{Y^{1+\theta}}\\
&\quad\leq | \chi_\varepsilon(\Diffb) \big( (( \chi_\varepsilon(\Diffb) + I) \zeta) ((\chi_\varepsilon(\Diffb) - I) \rho)\big)|_{L^2}\\
&\quad \leq  |(( \chi_\varepsilon(\Diffb) + I) \zeta) ((\chi_\varepsilon(\Diffb) - I) \rho) |_{L^2}\\
&\quad \leq  | ( \chi_\varepsilon(\Diffb) + I) \zeta|_\infty | (\chi_\varepsilon(\Diffb) - I) \rho |_{L^2} \\
&\quad \lesssim | ( \chi_\varepsilon(\Diffb) + I) \zeta|_{Y^{1+\theta}}| (\chi_\varepsilon(\Diffb) - I) \rho |_{L^2} \\
&\quad \leq 2| \zeta|_{Y^{1+\theta}} | (\chi_\varepsilon(\Diffb) - I) \rho |_{L^2},
\end{align*}
where the last line follows by the embedding $Y^{1+\theta}  \hookrightarrow C_\mathrm{b}({\mathbb R}^2)$.
To estimate \(| \chi_\varepsilon(\Diffb) - I) \zeta|_{L^2}\), note that
\[
\mathbb{R}^2 \setminus C_\varepsilon \subset \underbrace{\left\{ (k_1,k_2) \colon |k_1| > \frac{\delta}{\varepsilon}\right\}}_{\displaystyle= C_\varepsilon^1} \cup  \underbrace{\left\{ (k_1,k_2) \colon \left|\frac{k_2}{k_1}\right| > \frac{\delta}{\varepsilon}\right\}}_{\displaystyle=C_\varepsilon^2},
\]
so that
\begin{align*}
| (\chi_\varepsilon(\Diffb) - I) \zeta |_{L^2}^2 &= \int_{\mathbb{R}^2 \setminus C_\varepsilon} |\hat \zeta|^2 \dk \\
& \leq \int_{C_\varepsilon^1} |\hat \zeta|^2 \dk + \int_{C_\varepsilon^2} |\hat \zeta|^2 \dk\\
&\leq \frac{\varepsilon^2}{\delta^2} \int_{C_\varepsilon^1} k_1^2 |\hat \zeta|^2 \dk +  \frac{\varepsilon^2}{\delta^2} \int_{C_\varepsilon^2} \frac{k_2^2}{k_1^2} |\hat \zeta|^2 \dk \\
& \leq \frac{2\varepsilon^2}{\delta^2} | \zeta |_{Y^1}^2.\qedhere
\end{align*}
\end{proof}

\begin{lemma} \label{Difference of G 2}
Fix $\theta \in (0,1)$. The estimate
\[
| \tilde m(\Diffb)^{-1} ( \chi_\varepsilon(\Diffb) - I) (\zeta \rho) |_{Y^{1 +\theta}} \lesssim \varepsilon^{\frac{1-\theta}{2}} |\zeta|_{Y^1} |\rho|_{Y^1} \leq \varepsilon^{\frac{1-\theta}{2}} |\zeta|_{Y^{1+\theta}} |\rho|_{Y^{1+\theta}},
\]
holds for all \(\zeta, \rho \in Y^{1+\theta}\).
\end{lemma}
\begin{proof}
For \(\nu \in \{ k_1,  \frac{k_2}{k_1}\}\) we find that
\begin{align*}
&\bigg( 1 + k_1^2 + \frac{k_2^2}{k_1^2} \bigg)^{1 + \theta} \bigg( 1 + k_1^2 + \frac{k_2^2}{k_1^2} \bigg)^{-2} \bigg( \frac{\varepsilon}{\delta} | \nu| \bigg)^{1-\theta} = \left( \frac{\varepsilon}{\delta} \right)^{1-\theta} \Bigg( \frac{\left| \nu \right|}{1 + k_1^2 +\frac{k_2^2}{k_1^2}} \Bigg)^{1-\theta} \leq \frac{1}{2} \left( \frac{\varepsilon}{\delta} \right)^{1-\theta},
\end{align*}
so that
\begin{align*}
| \tilde m(\Diffb)^{-1} &( \chi_\varepsilon(\Diffb) - I) \zeta \rho|_{Y^{1+\theta}}^2\\
&\lesssim \int_{C_\varepsilon^1 \cup C_\varepsilon^2} \left( 1 + k_1^2 + \frac{k_2^2}{k_1^2}\right)^{1+\theta} \left( 1 + k_1^2 + \frac{k_2^2}{k_1^2}\right)^{-2} |\FF[\zeta \rho]|^2 \dk\\
&\lesssim \left( \frac{\varepsilon}{\delta} \right)^{1-\theta} \int_{C_\varepsilon^1} \left( 1 + k_1^2 + \frac{k_2^2}{k_1^2}\right)^{1+\theta} \left( 1 + k_1^2 + \frac{k_2^2}{k_1^2}\right)^{-2} |k_1|^{1-\theta} |\FF[\zeta \rho]|^2 \dk\\
&\qquad \mbox{}+ \left( \frac{\varepsilon}{\delta} \right)^{1-\theta} \int_{C_\varepsilon^2} \left( 1 + k_1^2 + \frac{k_2^2}{k_1^2}\right)^{1+\theta} \left( 1 + k_1^2 + \frac{k_2^2}{k_1^2}\right)^{-2} \left| \frac{k_2}{k_1} \right|^{1-\theta} |\FF[\zeta \rho]|^2 \dk\\
& \leq \left( \frac{\varepsilon}{\delta} \right)^{1-\theta} |\zeta \rho|_{L^2}^2  \\
& \lesssim \left( \frac{\varepsilon}{\delta} \right)^{1-\theta} |\zeta |_{L^4}^2 |\rho |_{L^4}^2 \\
&  \lesssim \left( \frac{\varepsilon}{\delta} \right)^{1-\theta} |\zeta |_{Y^1}^2 |\rho |_{Y^1}^2,
\end{align*}
where we have used Parseval's theorem, the Cauchy-Schwarz inequality and the embedding\linebreak
\(Y^1 \hookrightarrow L^4({\mathbb R}^2)\) (see Proposition~\ref{Y embeddings}).
\end{proof}

\begin{corollary} \label{Difference of G 3}
Fix $\theta \in (\frac{1}{2},1)$. The estimate
\[
\left| \tilde m(\Diffb)^{-1} \Big( \chi_\varepsilon(\Diffb) \big( ( \chi_\varepsilon(\Diffb) \zeta) (\chi_\varepsilon(\Diffb) \rho) \big) - \zeta \rho\Big)\right |_{Y^{1+\theta}} \lesssim \varepsilon^{\frac{1-\theta}{2}} |\zeta|_{Y^{1+\theta}} |\rho|_{Y^{1+\theta}}
\]
holds for all $\zeta$, $\rho \in Y^{1+\theta}$.
\end{corollary}
\begin{proof}
This result is obtained by writing
\begin{align*}
\tilde m(\Diffb)^{-1} & \Big( \chi_\varepsilon(\Diffb) \big( ( \chi_\varepsilon(\Diffb) \zeta) (\chi_\varepsilon(\Diffb) \rho) \big) - \zeta \rho\Big)\\
&=  \tfrac{1}{2} \tilde m(\Diffb)^{-1} \chi_\varepsilon(\Diffb) \big( ((\chi_\varepsilon(\Diffb) +1) \zeta) ( (\chi_\varepsilon(\Diffb) -1) \rho) \big)\\
&\qquad \mbox{}+ \tfrac{1}{2} \tilde m(\Diffb)^{-1} \chi_\varepsilon(\Diffb) \big( ((\chi_\varepsilon(\Diffb) +1) \rho) ( (\chi_\varepsilon(\Diffb) -1) \zeta) \big)\\ 
&\qquad \mbox{}+ \tilde m(\Diffb)^{-1} (\chi_\varepsilon(\Diffb) -1) (\zeta \rho),
\end{align*}
and applying Lemma \ref{Difference of G 1} to the first two terms on the right-hand side and Lemma \ref{Difference of G 2} to the third.
\end{proof}

It thus remains to show that
\[\diff_1\GG[\zeta_k^\star,0]=I+2\tilde{m}(\Diffb)^{-1}(\zeta_k^\star\cdot)\]
is an isomorphism; this fact {is a consequence of} the following result. \pagebreak

\begin{lemma}\label{lemma:compact}
The operator $\tilde{m}(\Diffb)^{-1}(\zeta_k^\star\cdot)\colon Y^{1+\theta} \rightarrow Y^{1+\theta}$ is compact.
\end{lemma}
\begin{proof}
Let $\{\zeta_j\}$ be a sequence which is bounded in $Y^1$.
We can find a subsequence of $\{\zeta_j\}$ (still denoted by $\{\zeta_j\}$)
which converges weakly in $L^2({\mathbb R}^2)$ (because $\{\zeta_j\}$ is bounded in $L^2({\mathbb R}^2)$)
and strongly in $L^2(|(x,y)|<n)$ for each $n \in {\mathbb N}$ (by Proposition \ref{Y embeddings}(ii) and a `diagonal'
argument). Denote the limit by $\zeta_\infty$. Since
\[
 |\zeta_k^\star\zeta_j-\zeta_k^\star\zeta_\infty|_{L^2(|(x,y)|<n)}\leq |\zeta_k^\star|_\infty |\zeta_j-\zeta_\infty|_{L^2(|(x,y)|<n)} \rightarrow 0
\]
as $j \rightarrow \infty$ for each $n \in {\mathbb N}$ and
\[
 \sup_j|\zeta_k^\star\zeta_j|_{L^2(|(x,y)|>n)}\leq \sup_{|(x,y)|>n}|\zeta_k^\star(x,y)|\sup_j|\zeta_j|_{L^2} \rightarrow 0
\]
as $n \rightarrow \infty$ we conclude that $\{\zeta_k^\star\zeta_j\}$ converges to $\zeta_k^\star\zeta_\infty$ as $j \rightarrow \infty$ in $L^2({\mathbb R}^2)$. It follows that
$\zeta \mapsto \zeta_k^\star \zeta$ is compact $Y^1 \rightarrow L^2({\mathbb R})$ and hence $Y^{1+\theta} \rightarrow L^2({\mathbb R})$;
the result follows from this fact and the observation that $\tilde{m}(\Diffb)^{-1}$ is continuous $L^2({\mathbb R}^2) \rightarrow Y^2 \hookrightarrow Y^{1+\theta}$.
\end{proof}

\begin{corollary}
The operator $I+2\tilde{m}(\Diffb)^{-1}(\zeta_k^\star\cdot)$ is an isomorphism $Y_\mathrm{e}^{1+\theta} \rightarrow Y_\mathrm{e}^{1+\theta}$.
\end{corollary}
\begin{proof}
The previous result shows that $I+2\tilde{m}(\Diffb)^{-1}(\zeta_k^\star\cdot)\colon Y_\mathrm{e}^{1+\theta} \rightarrow Y_\mathrm{e}^{1+\theta}$ is Fredholm with index $0$; it therefore remains to show that it is injective. {Suppose that $\zeta \in Y_\mathrm{e}^{1+\theta}$ satisfies
\begin{equation}
 \zeta+2\tilde{m}(\Diffb)^{-1}(\zeta_k^\star\zeta)=0. \label{Integral form}
\end{equation}
It follows that
\[
k_1\hat{\zeta}=\frac{-2 k_1^3}{ k_1^2+\frac{1}{2}(\beta-\frac{1}{3})k_1^4+k_2^2}\FF[\zeta_k^\star\zeta],
\qquad
k_2\hat{\zeta}=\frac{ -2 k_1^2 k_2}{ k_1^2+ \frac{1}{2}(\beta-\frac{1}{3})k_1^4+k_2^2}\FF[\zeta_k^\star\zeta]
\]
and hence $\zeta \in H^{j+1}({\mathbb R}^2)$ whenever $\zeta_k^\star\zeta \in H^j({\mathbb R}^2)$.
Since $\zeta \in L^2({\mathbb R}^2)$ and $\zeta \in H^j({\mathbb R}^2)$ implies $\zeta_k^\star\zeta \in H^j({\mathbb R}^2)$
because $\zeta_k^\star \in C_\mathrm{b}^j({\mathbb R}^2)$,
the space of smooth functions on ${\mathbb R}^2$ with bounded derivatives up to order $j$,
we find by bootstrapping that $\zeta \in H^\infty({\mathbb R}^2)$.}

Since $\zeta$ is smooth and satisfies \eqref{Integral form}, it satisfies the linear equation
\[\big((\beta-\tfrac{1}{3})\zeta_{xx}+2\zeta+2(\zeta_k^\star\zeta)\big)_{xx}-\zeta_{zz}=0,\]
and this equation has only the trivial smooth, decaying,
symmetric solution (see Lemma \ref{LWY results}(iii)).\
\end{proof}

Having completed the proof of Theorem~\ref{Final strong existence thm}, we now finalise the proof of Theorem \ref{Main result} by tracing back  the steps in the reduction procedure to construct solutions to \eqref{steady FDKP} which are uniformly approximated by a suitable scaling of $\zeta_k^\star$.

\begin{lemma}
The formula
\[u = u_1 + u_2(u_1),
\quad u_1(x,y) = \varepsilon^2 \zeta_k^\varepsilon(\varepsilon x,\varepsilon^2 y)\]
defines a smooth solution to the steady FDKP equation \eqref{steady FDKP} which satisfies the estimate
\[u(x,y)=\varepsilon^2 \zeta_k^\star(\varepsilon x,\varepsilon^2 y) + o(\varepsilon^2)\]
uniformly over $(x,y) \in {\mathbb R}^2$.
\end{lemma}\pagebreak
\begin{proof}
Theorem~\ref{Final strong existence thm} implies that
\[
|\zeta_k^\varepsilon-\zeta_k^\star|_\infty =o(1)
\]
as $\varepsilon \rightarrow 0$ because of the embedding \(Y^{1+\theta} \hookrightarrow C_\mathrm{b}(\mathbb{R}^2)\)
(see Proposition \ref{Continuity of Y}).
It follows that
\begin{align*}
u_1(x,y) & = \varepsilon^2 \zeta_k^\star(\varepsilon x,\varepsilon^2 y)
+\varepsilon^2 (\zeta_k^\varepsilon-\zeta_k^\star)(\varepsilon x,\varepsilon^2 y)  \\
& =  \varepsilon^2 \zeta_k^\star(\varepsilon x,\varepsilon^2 y) + o(\varepsilon^2)
\end{align*}
as $\varepsilon \rightarrow 0$ uniformly over $(x,y) \in {\mathbb R}^2$, while
\[|u_2(u_1)|_\infty \lesssim |u_2(u_1)|_{X_2} \lesssim \varepsilon | u_1 |_\varepsilon^2 \lesssim \varepsilon^3\]
because $|u_2(u_1)|_{X_2} \lesssim \varepsilon |u_1|_\varepsilon^2$ and $| u_1 |_\varepsilon = \varepsilon|\zeta|_{Y^1}$ with $\zeta \in B_M(0)\subseteq Y_\varepsilon^1$.
The fact that $u=u_1+u_2(u_1)$ is smooth follows from Proposition \ref{u is smooth}.
\end{proof}

%
%
%

\begin{Backmatter}

%
%

%
%
%


\section*{References}
\printbibliography[heading=none]

@preprint{GrovesWahlen25,
	author = {M. D. Groves and E. Wahl\'{e}n},
	note = {\emph{arXiv:2511.16843}},
	title = {Fully localised three-dimensional solitary water waves on {B}eltrami flows with strong surface tension},
	year = {2025}}

@article{LiuWangWeiYang26,
	author = {Y. Liu and Z. Wang and J. Wei and W. Yang},
	journal = {J. Math.\ Pures Appl.},
	pages = {103801},
	title = {From {KP-I} lump solution to travelling waves of {G}ross-{P}itaevskii equation},
	volume = {205},
	year = {2026}}

@article{Pelinovskii98,
	author = {D. E. Pelinovskii},
	journal = {J. Math.\ Phys.},
	pages = {5377-5395},
	title = {Rational solutions of the {K}adomtsev-{P}etviashvili hierarchy and the dynamics of their poles. {II}. {C}onstruction of the degenerate polynomial solutions},
	volume = {39},
	year = {1998}}

@article{Pelinovskii94,
	author = {D. E. Pelinovskii},
	journal = {J. Math.\ Phys.},
	pages = {5820--5830},
	title = {Rational solutions of the {K}adomtsev-{P}etviashvili hierarchy and the dynamics of their poles. {I}. {N}ew form of a general rational solution},
	volume = {35},
	year = {1994}}

@article{ClarksonDowie17,
	author = {P. A. Clarkson and E. Dowie},
	journal = {Trans.\ Math.\ Appl.},
	pages = {1--26},
	title = {Rational solutions of the {B}oussinesq equation and applications to rogue waves},
	volume = {1},
	year = {2017}}

@article{Clarkson08,
	author = {P. A. Clarkson},
	journal = {Anal.\ Appl.},
	pages = {349--369},
	title = {Rational solutions of the {B}oussinesq equation},
	volume = {4},
	year = {2008}}

@article{GalkinPelinovskiiStepanyants93,
	author = {V. M. Galkin and D. E. Pelinovskii and Y. A. Stepanyants},
	journal = {Physica D},
	pages = {246--255},
	title = {The structure of the rational solutions to the {B}oussinesq equation},
	volume = {80},
	year = {1995}}

@article{PelinkovskiiStepanyants93,
	author = {D. E. Pelinovskii and Y. A. Stepanyants},
	journal = {JETP Letters},
	pages = {24--28},
	title = {New multisoliton solutions of the {K}adomtsev-{P}etviashvili equation},
	volume = {57},
	year = {1993}}

@article{ManakovZakharovBordasItsMatveev77,
	author = {S. V. Manakov and V. E. Zakharov and L. A. Bordag and A. R. Its and V. B. Matveev},
	journal = {Phys.\ Lett.\ A},
	pages = {205--206},
	title = {Two-dimensional solitons of the {K}adomtsev-{P}etviashvili equation and their interaction},
	volume = {63},
	year = {1977}}

@preprint{GuiLaiLiuWeiYang25a,
	author = {C. Gui and S. Lai and Y. Liu and J. Wei and W. Yang},
	note = {\emph{arXiv:2509.06084}},
	title = {From {KP-I} lump solution to travelling wave of {3D} gravity-capillary water wave problem},
	year = {2025}}

@article{BuffoniGrovesWahlen22,
	author = {B. Buffoni and M. D. Groves and E. Wahl\'{e}n},
	journal = {J. Math.\ Fluid Mech.},
	pages = {55},
	title = {Fully localised three-dimensional gravity-capillary solitary waves on water of infinite depth},
	volume = {24},
	year = {2022}}

@article{LiuWeiYang24b,
	author = {Y. Liu and J. Wei and W. Yang},
	journal = {Physica D},
	pages = {134394},
	title = {Lump type solutions: {B}\"{a}cklund transformation and spectral properties},
	volume = {470},
	year = {2024}}

@article{LiuWeiYang24a,
	author = {Y. Liu and J. Wei and W. Yang},
	journal = {Proc.\ Lond.\ Math.\ Soc.},
	pages = {e12619},
	title = {Uniqueness of lump solution to the {KP-I} equation},
	volume = {129},
	year = {2024}}

@article{Groves21,
	author = {M. D. Groves},
	journal = {Water Waves},
	pages = {213--250},
	title = {An existence theory for gravity-capillary solitary water waves},
	volume = {3},
	year = {2021}}

@article{EhrnstroemGroves18,
	author = {M. Ehrnstr\"{o}m and M. D. Groves},
	journal = {Nonlinearity},
	pages = {5351--5384},
	title = {Small-amplitude fully localised solitary waves for the full-dispersion {K}adomtsev-{P}etviashvili equation},
	volume = {31},
	year = {2018}}

@article{StefanovWright20,
	author = {A. Stefanov and J. D. Wright},
	journal = {J. Dyn.\ Diff.\ Eqns.},
	pages = {85--99},
	title = {Small amplitude traveling waves in the full-dispersion {W}hitham equation},
	volume = {32},
	year = {2020}}

@article{ChironScheid18,
	author = {Chiron, D. and Scheid, C.},
	journal = {Nonlinearity},
	pages = {2809--2853},
	title = {Multiple branches of travelling waves for the {G}ross {P}itaevskii equation},
	volume = {31},
	year = {2018}}

@article{LiuWei19,
	author = {Liu, Y. and Wei, J.},
	journal = {Arch.\ Rat.\ Mech.\ Anal.},
	pages = {1335--1389},
	title = {Nondegeneracy, {M}orse index and orbital stability of the {KP-I} lump solution},
	volume = {234},
	year = 2019}

@article{LannesSaut14,
	author = {D. Lannes and J.-C. Saut},
	journal = {Kinet.\ Relat.\ Models},
	pages = {989--1009},
	title = {Remarks on the full dispersion {K}adomtsev-{P}etviashvili equation},
	volume = {6},
	year = {2014}}

@book{Lannes,
	address = {Providence, R.I.},
	author = {D. Lannes},
	number = {188},
	publisher = {American Mathematical Society},
	series = {Mathematical Surveys and Monographs},
	title = {The Water Waves Problem: Mathematical Analysis and Asymptotics},
	year = {2013}}

@article{deBouardSaut97b,
	author = {A. {de~Bouard} and J.-C. Saut},
	journal = {Ann.\ Inst.\ Henri Poincar\'{e} Anal.\ Non Lin\'{e}aire},
	pages = {211--236},
	title = {Solitary waves of generalized {K}adomtsev-{P}etviashvili equations},
	volume = {14},
	year = {1997}}

@article{DiasKharif99,
	author = {F. Dias and C. Kharif},
	journal = {Ann.\ Rev.\ Fluid Mech.},
	pages = {301--346},
	title = {Nonlinear gravity and capillary-gravity waves},
	volume = {31},
	year = {1999}}

\end{Backmatter}

\end{document}